\documentclass[11pt]{article}

\usepackage[T1]{fontenc}
\usepackage{lmodern}
\usepackage{microtype}
\usepackage{amsmath,amssymb,amsthm,mathtools}
\usepackage{xcolor}
\usepackage[margin=1.08in]{geometry}
\usepackage{enumitem}
\usepackage[
  backend=biber,
  style=alphabetic,
  sorting=nyt,
  giveninits=true,
  maxbibnames=99,
  maxalphanames=5,
  doi=true,
  isbn=false,
  url=false,
  eprint=true
]{biblatex}
\usepackage[
  colorlinks=true,
  linkcolor=magenta,
  citecolor=cyan,
  urlcolor=magenta,
  pdftitle={Kazhdan--Lusztig polynomials of matroids need not be unimodal},
  pdfauthor={Ronnie Cheng and Shurui Liu}
]{hyperref}

\setlist[enumerate]{leftmargin=2.2em,itemsep=2pt,topsep=4pt}
\setlist[itemize]{leftmargin=2em,itemsep=2pt,topsep=4pt}
\allowdisplaybreaks
\numberwithin{equation}{section}
\theoremstyle{plain}
\newtheorem{theorem}{Theorem}[section]
\newtheorem{proposition}[theorem]{Proposition}
\newtheorem{lemma}[theorem]{Lemma}
\newtheorem{corollary}[theorem]{Corollary}

\theoremstyle{definition}
\newtheorem{definition}[theorem]{Definition}

\theoremstyle{remark}
\newtheorem{remark}[theorem]{Remark}

\newcommand{\F}{\mathbb F}

\newcommand{\rk}{\operatorname{rk}}
\newcommand{\Span}{\operatorname{span}}
\newcommand{\qbinom}[2]{\genfrac{[}{]}{0pt}{}{#1}{#2}_{\!q}}
\newcommand{\card}[1]{\lvert #1\rvert}
\newcommand{\set}[1]{\left\{#1\right\}}
\newcommand{\suchthat}{\,\middle|\,}

\title{Kazhdan--Lusztig polynomials of matroids need not be unimodal}
\author{%
  \begin{tabular}{cc}
    Ronnie Cheng & Shurui Liu\\
    \href{mailto:rtcheng@stanford.edu}{\texttt{rtcheng@stanford.edu}}
      & \href{mailto:srliu@stanford.edu}{\texttt{srliu@stanford.edu}}
  \end{tabular}\\[1em]
  {\normalsize\textit{Department of Mathematics, Stanford University}}\\[0.4em]
  {\normalsize July 2026}
}
\date{}

\begin{document}
\maketitle
\begin{abstract}
We construct, over every finite field, representable matroids whose Kazhdan--Lusztig polynomials are not unimodal. In particular, the conjectures that all Kazhdan--Lusztig polynomials of matroids are log-concave and that they are real-rooted are both false.
Our examples are obtained by deleting points from finite projective geometries. More generally, we prove that, under a half-rank degree condition in each contraction quotient, the Kazhdan--Lusztig polynomial of every contraction enumerates the subspaces whose projective points lie in the corresponding deleted set, while the $Z$-polynomial agrees with that of the full projective geometry.
\end{abstract}

\tableofcontents

\section{Introduction}\label{sec:introduction}

For a loopless matroid $M$, let $L(M)$ be its lattice of flats. The Kazhdan--Lusztig polynomial $P_M(t)$ was introduced by Elias, Proudfoot, and Wakefield~\cite{EPW}. The family $M\mapsto P_M(t)$ is uniquely characterized by the following properties~\cite[Theorem~2.2]{BradenVysogorets}:
\begin{enumerate}[label=\textup{(\roman*)}]
\item $P_M(t)=1$ if $\rk M=0$;
\item $\deg P_M(t)<\rk M/2$ if $\rk M>0$;
\item for every matroid $M$, the polynomial
\[
  Z_M(t)=\sum_{F\in L(M)}
  t^{\rk_M(F)}P_{M/F}(t)
\]
is palindromic of degree $\rk M$.
\end{enumerate}
Here $M/F$ denotes the contraction by $F$. The polynomial $Z_M(t)$ in \textup{(iii)}, introduced by Proudfoot, Xu, and Young~\cite[Definition~2.1]{PXY}, is called the $Z$-polynomial of $M$. is called the $Z$-polynomial of $M$. For representable matroids, $P_M(t)$ and $Z_M(t)$ record, respectively, local and global intersection cohomology of varieties attached to a realization~\cite{EPW,PXY}. Their coefficients are nonnegative for all matroids by the singular Hodge theory of Braden, Huh, Matherne, Proudfoot, and Wang~\cite{BHMPW}.

Elias, Proudfoot, and Wakefield conjectured that the coefficients of $P_M(t)$ form a log-concave sequence with no internal zeros~\cite[Conjecture~2.5]{EPW}, and Gedeon, Proudfoot, and Young conjectured that all its zeros are real and negative~\cite[Conjecture~3.2]{GPY}. These conjectures were supported by formulas and computations for many families. Kazhdan--Lusztig polynomials have been studied for uniform and $q$-niform matroids, thagomizer matroids, fan, wheel, and whirl matroids, $\rho$-removed uniform and sparse paving matroids, paving matroids, braid matroids, and Dowling geometries~\cite{GPYEquivariant,GaoUniform,ProudfootQUniform,GedeonThagomizer,LuXieYang,LeeNasrRadcliffeRemoved,LeeNasrRadcliffeSparse,FerroniNasrVecchi,KarnNasrProudfootVecchi,FerroniLarsonBraid,FerroniLarsonDowling}. Real-rootedness was proved for fan, wheel, and whirl matroids and for several families of uniform matroids~\cite{LuXieYang,GaoUniform}; log-concavity was proved for uniform, $q$-niform, and thagomizer matroids~\cite{XieZhangUniform,GaoLiXieYangZhang,WuZhangThagomizer}.

Recall that a finite sequence is \emph{unimodal} if it weakly
increases up to some index and then weakly decreases. We show that, in general, Kazhdan--Lusztig polynomials of matroids need not even be unimodal.

Let $V=\F_q^r$, and let $\mathbb P(V)$ be the set of one-dimensional subspaces of $V$. A subset $E\subseteq\mathbb P(V)$ represents a matroid: the rank of a set of projective points is the dimension of their linear span. Let
$M$ be the matroid represented by
\[
  E=\mathbb P(V)\setminus D,
\]
where $D$ is a set of deleted points. 

We say a subspace $U\le V$ is a \emph{flat span} if $E\cap\mathbb P(U)$ spans $U$; the corresponding flat is $F_U=E\cap\mathbb P(U)$. Conversely, every flat $F$ of $M$ is $F_U$ for the unique flat span $U=\Span(F)$. For every flat span $U$, define the \emph{subspace polynomial} with respect to $U$ by
\begin{equation}\label{eq:intro-CU}
  C_U(t)=
  \sum_{\substack{U\le A\le V\\
        \mathbb P(A)\setminus\mathbb P(U)\subseteq D}}
  t^{\dim A-\dim U}.
\end{equation}
In particular, $C_0(t)$ enumerates the subspaces whose projective points all lie in $D$.

Our main result identifies these subspace polynomials with Kazhdan--Lusztig polynomials when the deleted set $D$ is ``small''.

\begin{theorem}\label{thm:main-intro}
Assume that $E$ spans $V$ and that
\[
  \deg C_U(t)<\frac{r-\dim U}{2}
\]
for every flat span $U < V$. Then
\[
  P_{M/F_U}(t)=C_U(t)
\]
for every flat $F_U$ of $M$. Moreover,
\[
  Z_M(t)=\sum_{j=0}^{r}\qbinom{r}{j}t^j,
\]
where $\qbinom{r}{j}$ is the Gaussian binomial coefficient.
\end{theorem}

The hypothesis says that no deleted subspace in any contraction quotient associated with a flat span $U < V$ reaches half the quotient dimension. Section~\ref{sec:geometry} identifies it with smallness of a natural map between matroid Schubert varieties.

Theorem~\ref{thm:main-intro} yields the following family of Kazhdan--Lusztig polynomials.

\begin{corollary} \label{cor:realization-intro}
Let $q$ be a prime power and let $k$ and $s$ be nonnegative integers. For every sufficiently large $r$, there is an $\F_q$-representable matroid $M$ of rank $r$ such that
\[
  P_M(t)=\sum_{j=0}^{k}\qbinom{k}{j}t^j+s t.
\]
\end{corollary}

For each $q$, taking $k=6$ and
\[
  s=\qbinom{6}{2}-\qbinom{6}{1}+1
\]
gives a nonunimodal polynomial and hence a counterexample over $\F_q$.

\begin{corollary}\label{cor:nonunimodal-intro}
For every finite field $\F_q$, there is an $\F_q$-representable matroid whose Kazhdan--Lusztig polynomial is not unimodal. Consequently, Kazhdan--Lusztig polynomials of matroids need not be log-concave or real-rooted.
\end{corollary}

Indeed, a positive log-concave sequence is unimodal, while Newton's inequalities imply that the coefficient sequence of a real-rooted polynomial with nonnegative coefficients is log-concave.

For $q=2$, $k=6$, and $s=589$, the construction gives the following polynomial for a matroid representable over $\F_2$:
\[
  1+652t+651t^2+1395t^3+651t^4+63t^5+t^6.
\]
Its coefficients decrease from $652$ to $651$ and then increase to $1395$. The construction may be carried out in rank $17$, yielding a binary matroid on $2^{17}-1-652=130{,}419$ elements.

Section~\ref{sec:projective-deletions} proves Theorem~\ref{thm:main-intro} and deduces Corollaries~\ref{cor:realization-intro} and~\ref{cor:nonunimodal-intro}.
Section~\ref{sec:geometry} gives a geometric proof using a small map between matroid Schubert varieties.

\section*{Acknowledgements}

The initial counterexample to the real-rootedness conjecture was discovered by Rethlas~\cite{Rethlas}, an AI agent for mathematical research, developed by the Rethlas team, namely Haocheng Ju, Jiedong Jiang, Shurui Liu, Guoxiong Gao, Yuefeng Wang, Zeming Sun, Leheng Chen, Bin Wu, led by Professor Liang Xiao and Professor Bin Dong. The details of AI usage are documented in Appendix~\ref{app:ai} for transparency. The authors would like to thank Haocheng Ju, who kindly helped with the experiments. The authors are very grateful to Matt Larson for many insightful comments that substantially improved the results and the exposition.

\section{Projective deletions}\label{sec:projective-deletions}

We retain the notation of the introduction. Thus $V=\F_q^r$, $D\subseteq\mathbb P(V)$, $E=\mathbb P(V)\setminus D$, and $M$ is the matroid represented by $E$. We write
\[
  [m]_q=1+q+\cdots+q^{m-1}
\]
for the number of points of $\mathbb P(\F_q^m)$ with the convention $[0]_q=0$, and 
\[
  \qbinom{m}{j}
  =\prod_{i=0}^{j-1}\frac{q^{m-i}-1}{q^{j-i}-1}
\]
for the number of $j$-dimensional subspaces of $\F_q^m$.
\subsection{Subspace polynomials and Kazhdan--Lusztig polynomials}
For a subspace $A\le V$, define its \emph{surviving span} by \[
    \sigma(A)=\Span\bigl(E\cap\mathbb P(A)\bigr).
\]

Recall that a subspace $U\le V$ is a flat span if $E\cap\mathbb P(U)$ spans $U$. In this case the corresponding flat is $F_U=E\cap\mathbb P(U)$.

\begin{lemma}\label{lem:surviving-span-body}
The subspace $\sigma(A)$ is a flat span. If $U$ is a flat span, then
\[
  \sigma(A)=U
  \quad\Longleftrightarrow\quad
  U\le A\ \text{ and }\
  \mathbb P(A)\setminus\mathbb P(U)\subseteq D.
\]
\end{lemma}
\begin{proof}
The subspace $\sigma(A)$ is a flat span, spanned by $E\cap\mathbb P(\sigma(A))=E\cap\mathbb P(A)$.

If $\sigma(A)=U$, then $U\le A$ and every surviving point of
$\mathbb P(A)$ lies in $\mathbb P(U)$, so
$\mathbb P(A)\setminus\mathbb P(U)\subseteq D$. Conversely, if
$U\le A$ and
$\mathbb P(A)\setminus\mathbb P(U)\subseteq D$, then
\[
  E\cap\mathbb P(A)=E\cap\mathbb P(U)=F_U.
\]
Since $F_U$ spans $U$, one has $\sigma(A)=U$.
\end{proof}
 
We can enumerate all subspaces by their surviving spans.

\begin{proposition}\label{prop:surviving-partition}
One has
\begin{equation}\label{eq:bivariate-C}
  \sum_{F_U\in L(M)}x^{\dim U}C_U(y)
  =
  \sum_{A\le V}x^{\dim\sigma(A)}y^{\dim A-\dim\sigma(A)}.
\end{equation}
Consequently,
\begin{equation}\label{eq:Gaussian-C}
  \sum_{F_U\in L(M)}t^{\dim U}C_U(t)
  =
  \sum_{A\le V}t^{\dim A}
  =
  \sum_{j=0}^{r}\qbinom{r}{j}t^j.
\end{equation}
\end{proposition}

\begin{proof}
Expand the left-hand side of \eqref{eq:bivariate-C} using \eqref{eq:intro-CU}. By Lemma~\ref{lem:surviving-span-body}, a subspace $A$ occurs in exactly one summand, namely the summand indexed by $U=\sigma(A)$, and it contributes
\[
  x^{\dim\sigma(A)}y^{\dim A-\dim\sigma(A)}.
\]
This proves \eqref{eq:bivariate-C}. Setting $x=y=t$ gives \eqref{eq:Gaussian-C}.
\end{proof}

Let $U$ be a flat span. The quotient map induces
\[
  \pi_U:\mathbb P(V)\setminus\mathbb P(U)
  \longrightarrow \mathbb P(V/U),
  \qquad
  L\longmapsto (L+U)/U.
\]
Define
\begin{equation}\label{eq:DU-body}
  D_U=
  \set{\overline L\in\mathbb P(V/U)\suchthat
        \pi_U^{-1}(\overline L)\subseteq D}.
\end{equation}
A quotient direction lies in $D_U$ exactly when its entire projective fiber is contained in $D$. Thus, up to simplification, $M/F_U$ is again a projective deletion,
with deleted set $D_U$. 

\begin{lemma}\label{lem:contraction-body}
For every flat span $U$, the simplification of $M/F_U$ is the matroid corresponding to $\mathbb P(V/U)\setminus D_U$. If $U\le W$ are flat spans, then
\[
  (D_U)_{W/U}=D_W.
\]
\end{lemma}
\begin{proof}
The contraction $M/F_U$ is represented in $V/U$ by the images of the surviving points outside $\mathbb P(U)$. A quotient direction is therefore missing exactly when its entire fiber lies in $D$, which is the definition of $D_U$.

Quotienting first by $U$ and then by $W/U$ is the same as quotienting directly by $W$. A direction modulo $W$ disappears after the two quotients exactly when every original point in its fiber outside $\mathbb P(W)$ belongs to $D$.
\end{proof}
Since simplification preserves the ranked lattice of flats, it does not change the Kazhdan--Lusztig or $Z$-polynomial. Furthermore,
\[
  C_U(t)
  =
  \sum_{\substack{\overline A\le V/U\\
        \mathbb P(\overline A)\subseteq D_U}}
  t^{\dim\overline A},
\]
and $C_U$ is obtained by applying the definition of $C_0$ to the deleted set $D_U$ in the quotient $V/U$.

\begin{proof}[Proof of Theorem~\ref{thm:main-intro}]
We argue by induction on $r$. The rank-zero case is immediate. Let $U\ne0$ be a flat span. By Lemma~\ref{lem:contraction-body}, up to simplification the contraction $M/F_U$ is the projective deletion associated with $D_U$ in $V/U$. The identity
\[
  (D_U)_{W/U}=D_W
\]
shows that the half-rank degree hypotheses are inherited by this contraction.
The induction hypothesis therefore gives
\[
  P_{M/F_U}(t)=C_U(t)
\]
for every nonzero flat span $U$.

Proposition~\ref{prop:surviving-partition} now gives
\[
  C_0(t)
  +
  \sum_{\substack{F_U\in L(M)\\ U\ne0}}
  t^{\dim U}P_{M/F_U}(t)
  =
  \sum_{j=0}^{r}\qbinom{r}{j}t^j.
\]
The right-hand side is palindromic of degree $r$, while
$\deg C_0<r/2$. The palindromic characterization of the
Kazhdan--Lusztig polynomial therefore gives
\[
  P_M(t)=C_0(t).
\]
Substituting $P_{M/F_U}=C_U$ into \eqref{eq:Gaussian-C} gives the
formula for $Z_M(t)$.
\end{proof}

\begin{remark}
Theorem~\ref{thm:main-intro} also admits a proof from the original Kazhdan--Lusztig recursion. Here $\chi_N(t)$ denotes the characteristic polynomial of a matroid $N$, and $M|F$ denotes the restriction of $M$ to $F$:
\[
  t^{\rk M}P_M(t^{-1})
  =
  \sum_{F\in L(M)}
  \chi_{M|F}(t)P_{M/F}(t).
\]
For every projective deletion, the finite-field method gives the parallel identity
\[
  t^rC_0(t^{-1})
  =
  \sum_{F_U\in L(M)}
  \chi_{M|F_U}(t)C_U(t).
\]
After setting $t=q^m$, the right-hand side counts pairs $(A,\psi)$, where $A\le V$ and $\psi:A\longrightarrow\F_{q^m}$ is $\F_q$-linear and nonzero on $\mathbb{P}(A) \cap E$. 
Grouping these pairs by $K=\ker\psi$, a fixed subspace $K$ with $\mathbb P(K)\subseteq D$ contributes $q^{m(r-\dim K)}$, and summing over $K$ gives $q^{mr}C_0(q^{-m})$.
Applying the same identity in every quotient and comparing inductively with the Kazhdan--Lusztig recursion proves Theorem~\ref{thm:main-intro} under the half-rank degree condition.
\end{remark}

\subsection{Counterexamples to unimodality}

We now construct deletion sets for which the half-rank degree condition on $C_U$ is easy to verify.

\begin{definition}\label{def:line-separated-body}
Let $W\le V$ and let $S\subseteq\mathbb P(V)\setminus\mathbb P(W)$. We say that $S$ is \emph{line-separated from $W$} if every projective line contained in $\mathbb P(W)\cup S$ is contained in $\mathbb P(W)$.
\end{definition}

\begin{proposition}\label{prop:line-separated-body}
Let $\dim V=r$, let $0\le k=\dim W<r/2$, and let $S$ be line-separated from $W$. Put
\[
  s=\card S,
  \qquad
  D=\mathbb P(W)\cup S.
\]
If
\begin{equation}\label{eq:size-bound-body}
  \card{D} = [k]_q+s<q^{r-2},
\end{equation}
then the matroid represented by $\mathbb P(V)\setminus D$ has Kazhdan--Lusztig polynomial
\begin{equation}\label{eq:line-separated-body}
  \sum_{j=0}^{k}\qbinom{k}{j}t^j+s t.
\end{equation}
\end{proposition}

\begin{proof}
A subspace of dimension at least two whose projective points lie in $D$ is contained in $W$: otherwise it contains a projective line lying in $D$ but not in $\mathbb P(W)$. Hence
\[
  C_0(t)=\sum_{j=0}^{k}\qbinom{k}{j}t^j+s t.
\]
This polynomial has degree $k < r/2$ unless $k=0$ and $s>0$, where the size condition forces $r\ge3$, and then $\deg C_0=1<r/2$.

Let $U\ne0$ be a proper flat span. Since every point of $\mathbb P(W)$ was deleted, one has $U\not\le W$. Suppose that $\deg C_U\ge2$. Then there is a subspace $B > U$ such that
\[
  \dim B=\dim U+2,
  \qquad
  \mathbb P(B)\setminus\mathbb P(U)\subseteq D.
\] Choose a complement $L=\langle \ell_1,\ell_2\rangle$ to $U$ in $B$. Since $\mathbb P(L)\subseteq D$, line separation gives $L\le W$. Choose $u\in U\setminus W$. The subspace
\[
  L'=\langle \ell_1+u,\ell_2\rangle
\]
is another complement to $U$ in $B$, so $\mathbb P(L')\subseteq D$; but $L'\not\le W$, a contradiction. Thus $\deg C_U\le1$.

If $r-\dim U\ge3$, this gives the strict half-rank degree bound. If $r-\dim U\le2$ and $C_U$ were nonconstant, then $D$ would contain an entire one-dimensional quotient fiber, consisting of at least $q^{r-2}$ points, contrary to \eqref{eq:size-bound-body}. Finally, the surviving points span $V$, since otherwise $D$ would contain the $q^{r-1}$ points outside a hyperplane. Theorem~\ref{thm:main-intro} now gives \eqref{eq:line-separated-body}.
\end{proof}

\begin{proof}[Proof of Corollary~\ref{cor:realization-intro}]
Choose $r$ large enough that
\[
  k<\frac r2,
  \qquad
  [k]_q+s<q^{r-2},
  \qquad
  s\le q^{r-k-1}.
\]
Write $V=W\oplus X$ with $\dim W=k$, and choose a hyperplane $H<X$. The affine chart $\mathbb P(X)\setminus\mathbb P(H)$ has $q^{r-k-1}$ points, so it contains a subset $S$ of cardinality $s$. 

We now check that the set $S$ is line-separated from $W$. If a line contained in $\mathbb P(W)\cup S$ is not contained in $\mathbb P(W)$, then it meets $\mathbb P(W)$ in at most one point and therefore contains at least two points of $S$. A line through two points of $S$ meets $\mathbb P(H)$, so it is not contained in $\mathbb P(W)\cup S$. A line through a point of $S$ and a point of $\mathbb P(W)$ contains points with nonzero components in both $W$ and $X$, and is again not contained in $\mathbb P(W)\cup S$. Therefore, Proposition~\ref{prop:line-separated-body} gives the required polynomial.
\end{proof}

\begin{proof}[Proof of Corollary~\ref{cor:nonunimodal-intro}]
Take $k=6$ and
\[
  s=\qbinom{6}{2}-\qbinom{6}{1}+1.
\]
By Corollary~\ref{cor:realization-intro}, the coefficients of $t$, $t^2$, and $t^3$ in the resulting polynomial are
\[
  \qbinom{6}{2}+1,
  \qquad
  \qbinom{6}{2},
  \qquad
  \qbinom{6}{3},
\]
respectively. Since $\qbinom{6}{3}>\qbinom{6}{2}$, these coefficients form a strict local valley, so the polynomial is not unimodal.
\end{proof}

\begin{remark}
Let $W\le V$ have dimension $k<r$, and let $B_{r,k}(q)$ be the matroid
represented by
\[
  \mathbb P(V)\setminus\mathbb P(W).
\]
These matroids are often called Bose--Burton geometries
\cite{BoseBurton,Bonin}. For $m\ge0$, set
\[
  G_{m,q}(t)=\sum_{j=0}^{m}\qbinom{m}{j}t^j,
\]
and adopt the convention $\qbinom{k}{j}=0$ if $j<0$ or $j>k$.

Every nonzero flat span $U$ is not contained in $W$, and one checks
that $D_U=\varnothing$. Thus every nonempty contraction of
$B_{r,k}(q)$ simplifies to a projective geometry and has
Kazhdan--Lusztig polynomial $1$. Consequently,
\[
  Z_{B_{r,k}(q)}(t)
  =
  P_{B_{r,k}(q)}(t)+G_{r,q}(t)-G_{k,q}(t).
\]
Palindromicity of the left-hand side and the Kazhdan--Lusztig degree
bound give
\[
  P_{B_{r,k}(q)}(t)
  =
  \sum_{0\le i<r/2}
  \left(
    \qbinom{k}{i}-\qbinom{k}{r-i}
  \right)t^i.
\]
In particular, if $k<r/2$, then
\[
  P_{B_{r,k}(q)}(t)=G_{k,q}(t),
  \qquad
  Z_{B_{r,k}(q)}(t)=G_{r,q}(t).
\] The constructions above are perturbations of these Bose--Burton geometries by further line-separated point deletions. 
\end{remark}

\section{The geometric argument}\label{sec:geometry}
Section \ref{sec:projective-deletions} gives a combinatorial proof of Theorem~\ref{thm:main-intro}. We now explain the geometry behind the statement. The main construction is a proper map obtained by forgetting the coordinates indexed by $D$. Its fibers are paved by the affine spaces counted by the polynomials $C_U(t)$, and the half-rank degree condition says exactly that this map is small.

Set $\Bbbk=\overline{\F_q}$ and $V_{\Bbbk}=V\otimes_{\F_q}\Bbbk$. Fix a prime
$\ell\ne\operatorname{char}\Bbbk$. We use $\ell$-adic \'{e}tale cohomology and intersection cohomology with coefficients in $\overline{\mathbb Q}_{\ell}$.

Suppose that $E\subseteq\mathbb P(V)$ spans $V$, and let $M$ be the
matroid represented by $E$. Choose a nonzero representative
$v_e\in e$ for every $e\in E$. Evaluation gives an embedding
\[
  V_{\Bbbk}^{\vee}
  \longrightarrow
  \mathbb A_{\Bbbk}^{E},
  \qquad
  \phi\longmapsto\bigl(\phi(v_e)\bigr)_{e\in E}.
\]
Let $L_E$ be its image. The associated matroid Schubert variety is
\[
  Y_M=\overline{L_E}\subseteq(\mathbb P^1_{\Bbbk})^E.
\]
These compactifications were studied by Ardila and Boocher~\cite{ArdilaBoocher}; their relation to the $Z$-polynomial was established by Proudfoot, Xu, and Young~\cite{PXY}.

We recall the part of the boundary stratification that will be used. For $y=(y_e)_{e\in E}\in Y_M$, set
\[
  F(y)=\set{e\in E\suchthat y_e\ne\infty}.
\]
This is a flat of $M$. If $U$ is a flat span, define
\[
  \mathcal S_U
  =
  \set{y\in Y_M\suchthat F(y)=F_U}.
\]
Then
\[
  Y_M=\bigsqcup_{F_U\in L(M)}\mathcal S_U,
  \qquad
  \mathcal S_U\cong U_{\Bbbk}^{\vee}\cong\mathbb A_{\Bbbk}^{\dim U};
\]
see \cite[Lemmas~7.5 and 7.6]{PXY}. In particular, $\mathcal S_V=L_E$ is the dense affine open stratum, and $\mathcal S_0$ is the point all of whose coordinates are $\infty$.

Let $\operatorname{IC}_{Y_M}$ be the middle-perversity intersection complex. The cohomology of its stalk at a point $y\in\mathcal S_U$, with the standard shift, is the local intersection cohomology along that stratum. The geometric interpretations of the two matroid polynomials are
\begin{equation}\label{eq:local-global-IH-body}
  P_{M/F_U}(t)
  =
  \sum_{i\ge0}\dim\operatorname{IH}^{2i}_y(Y_M)t^i,
  \qquad
  Z_M(t)
  =
  \sum_{i\ge0}\dim\operatorname{IH}^{2i}(Y_M)t^i,
\end{equation}
with odd local and global intersection cohomology equal to zero; see \cite[Theorem~3.10]{EPW} and \cite[Lemmas~7.5--7.7 and Theorems~7.1 and~7.2]{PXY}.

Return to the deletion $E=\mathbb P(V)\setminus D$. Let $\widetilde M$ be the matroid realized by $\mathbb{P}(V)$. Forgetting the coordinates indexed by $D$ gives a proper morphism
\[
  \pi:Y_{\widetilde M}\longrightarrow Y_M.
\]
Its image is closed and contains the dense subset $L_E$, so it is all of $Y_M$. The map restricts to an isomorphism from the dense affine open $L_{\mathbb P(V)}$ onto $L_E$, because both evaluation spaces are naturally isomorphic to $V_{\Bbbk}^{\vee}$. Thus $\pi$ is birational.

For a subspace $A\le V$, let $\widetilde{\mathcal S}_A$ be the stratum of $Y_{\widetilde M}$ indexed by $\mathbb P(A)$. Then
\[
  \widetilde{\mathcal S}_A\cong A_{\Bbbk}^{\vee}.
\]
After forgetting the deleted coordinates, this stratum maps to the stratum indexed by the surviving span $\sigma(A)$. If $U=\sigma(A)$, the map between the two strata is the restriction map
\[
  A_{\Bbbk}^{\vee}\longrightarrow U_{\Bbbk}^{\vee}.
\]

Recall that an affine paving of a variety $X$ is a filtration
\[
  \varnothing=X_{-1}\subseteq X_0\subseteq\cdots\subseteq X_N=X
\]
by closed subvarieties such that every difference $X_i\setminus X_{i-1}$ is an affine space.

\begin{proposition}\label{prop:fiber-paving-body}
Let $U$ be a flat span and let $y\in\mathcal S_U$. The reduced fiber $(\pi^{-1}(y))_{\mathrm{red}}$ admits an affine paving whose cells are indexed by the subspaces $A\le V$ satisfying
\[
  U\le A,
  \qquad
  \mathbb P(A)\setminus\mathbb P(U)\subseteq D.
\]
The cell indexed by $A$ has dimension $\dim A-\dim U$. Consequently,
\[
  \sum_{i\ge0}\dim H^{2i}\bigl(\pi^{-1}(y),\overline{\mathbb Q}_{\ell}\bigr)t^i
  =C_U(t),
\]
the odd cohomology of the fiber vanishes, and
\[
  \dim\pi^{-1}(y)=\deg C_U(t).
\]
\end{proposition}

\begin{proof}
The intersection $\widetilde{\mathcal S}_A\cap(\pi^{-1}(y))_{\mathrm{red}}$ is nonempty exactly when $\sigma(A)=U$. By Lemma~\ref{lem:surviving-span-body}, this is equivalent to
\[
  U\le A,
  \qquad
  \mathbb P(A)\setminus\mathbb P(U)\subseteq D.
\]
For such an $A$, the intersection is a fiber of the surjective restriction map $A_{\Bbbk}^{\vee}\to U_{\Bbbk}^{\vee}$, and hence is an affine space of dimension $\dim A-\dim U$.

The closure of $\widetilde{\mathcal S}_A$ is the union of the strata $\widetilde{\mathcal S}_B$ with $B\le A$. Choose a linear extension of inclusion on the relevant subspaces, with proper subspaces preceding the subspaces that contain them. The successive unions of the corresponding strata are closed. Intersecting this filtration with $(\pi^{-1}(y))_{\mathrm{red}}$ gives the required affine paving.

A $d$-dimensional affine cell contributes one copy of $\overline{\mathbb Q}_{\ell}(-d)$ in compactly supported cohomological degree $2d$ and nothing in odd degree. Passing to the reduction changes neither dimension nor $\ell$-adic cohomology. Since $\pi$ is proper, its fibers are proper, so ordinary and compactly supported cohomology agree. The cells are indexed by exactly the subspaces counted by $C_U(t)$.
\end{proof}

We now impose the half-rank degree condition. Since $\dim Y_M=r$ and $\dim\mathcal S_U=\dim U$, Proposition~\ref{prop:fiber-paving-body} identifies
\[
  \deg C_U(t)<\frac{r-\dim U}{2}
\]
with
\[
  2\dim\pi^{-1}(y)
  <
  \operatorname{codim}_{Y_M}\mathcal S_U
\]
for every nonopen stratum. These are precisely the smallness inequalities for $\pi$.

Every contraction of $\widetilde M$ simplifies to a projective geometry, and hence has Kazhdan--Lusztig polynomial $1$~\cite[Proposition~2.14]{EPW}. By \eqref{eq:local-global-IH-body}, the source $Y_{\widetilde M}$ is rationally smooth, so
\[
  \operatorname{IC}_{Y_{\widetilde M}}
  \cong
  \overline{\mathbb Q}_{\ell}[r].
\]
Let $j:L_E\hookrightarrow Y_M$ denote the inclusion of the dense open
stratum. Since $\pi$ is proper and small and restricts to an isomorphism
over $L_E$, the $\ell$-adic decomposition theorem and its standard
small-map consequence~\cite[Theorem~6.2.5]{BBD} give
\[
  R\pi_*\operatorname{IC}_{Y_{\widetilde M}}
  \cong
  j_{!*}\overline{\mathbb Q}_{\ell}[r]
  =
  \operatorname{IC}_{Y_M}.
\]
Taking the stalk at $y\in\mathcal S_U$ and applying proper base change identifies the local intersection cohomology of $Y_M$ at $y$ with the ordinary cohomology of $\pi^{-1}(y)$. Equations \eqref{eq:local-global-IH-body} and Proposition~\ref{prop:fiber-paving-body} therefore give
\[
  P_{M/F_U}(t)=C_U(t).
\]
This is the geometric proof of the local assertion in Theorem~\ref{thm:main-intro}.

Taking global cohomology in the same pushforward isomorphism gives
\[
  \operatorname{IH}^{i}(Y_M)
  \cong
  H^i\bigl(Y_{\widetilde M},\overline{\mathbb Q}_{\ell}\bigr).
\]
The strata of $Y_{\widetilde M}$ form an affine paving, with one cell of dimension $\dim A$ for every subspace $A\le V$. Hence
\[
  Z_M(t)
  =
  \sum_{A\le V}t^{\dim A}
  =
  \sum_{j=0}^{r}\qbinom{r}{j}t^j,
\]
which is the geometric proof of the global assertion in Theorem~\ref{thm:main-intro}.

\appendix
\refstepcounter{section}\label{app:ai}
\section*{Appendix: Details of AI Usage}
\addcontentsline{toc}{section}{Appendix: Details of AI Usage}

The counterexample that motivated this paper, obtained using the projective-deletion method, was first produced by the automated mathematical reasoning agent Rethlas~\cite{Rethlas}, developed by the Rethlas team, namely Haocheng Ju, Jiedong Jiang, Shurui Liu, Guoxiong Gao, Yuefeng Wang, Zeming Sun, Leheng Chen, Bin Wu, led by Professor Liang Xiao and Professor Bin Dong. 

In this appendix, for transparency we record the experiment in detail: the exact problem we posed (\S\ref{app:ai-prompt}), the output accepted by the Rethlas verification agent (\S\ref{app:ai-output}), and the discovery trajectory that led to it (\S\ref{app:ai-trajectory}).

Note that Lemma~\ref{app:lem-count} is not a new result and it is already in literature (with different notations). It is essentially \cite[Theorem 2.2]{Athanasiadis}, and the idea dates back to \cite[\S 16]{CrapoRota}.

\paragraph{Summary}
We ran Rethlas with base model \texttt{GPT-5.6 Sol} and reasoning effort set to
\texttt{max}. The experiment was carried out on July 13, 2026 and took 4 hours and 59 minutes to produce a solution accepted by the Rethlas verification agent. For comparison we gave the same problem to GPT-5.6 Sol Ultra through its web interface; across two prompts, it failed to give a complete solution.


\subsection{Problem prompt to Rethlas}\label{app:ai-prompt}

\begin{quote}\itshape
For each matroid $M$, define its Kazhdan--Lusztig polynomial
$P_M(t)\in\mathbb Z[t]$ as the unique family satisfying: $P_M(t)=1$ when
$\rk M=0$; $\deg P_M<\rk(M)/2$ when $\rk M>0$; and
$t^{\rk M}P_M(t^{-1})=\sum_{F\in L(M)}\chi_{M|F}(t)\,P_{M/F}(t)$,
where $L(M)$ is the lattice of flats, $M|F$ and $M/F$ are restriction and
contraction, and $\chi$ is the characteristic polynomial. Prove or disprove
that all zeros of $P_M(t)$ lie on the negative real axis.
\end{quote}

\subsection{Output of Rethlas}\label{app:ai-output}

The following is a typeset version of the final output returned by Rethlas.
Only the notation, \LaTeX{} formatting, labels, and cross-references have been adapted;
the mathematical content is unchanged.

\begin{lemma}[finite-field characteristic count]\label{app:lem-count}
Let $E$ be a set of projective points spanning an $n$-dimensional
$\F_q$-space $X$, and let $M(E)$ be its vector matroid. If $Q=q^m$, then

\begin{equation}\label{app:eq-count}
\chi_{M(E)}(Q)=\#\{\phi:X\to\F_Q:\phi(x)\ne0\text{ for every }x\in E\}.
\end{equation}
Moreover, if $I_h(Q)=\prod_{j=0}^{h-1}(Q-q^j)$, then for every $n\ge0$,
\begin{equation}\label{app:eq-QN}
Q^n=\sum_{h=0}^n\qbinom{n}{h}I_h(Q).
\end{equation}
\end{lemma}

\begin{proof}
The subset expansion of the characteristic polynomial is
\[
  \chi_{M(E)}(Q)
  =
  \sum_{A\subseteq E}(-1)^{|A|}Q^{n-\rk_{M(E)}(A)}.
\]
Inclusion--exclusion
gives the same expression for the number of linear maps $X\to\F_Q$ vanishing on
none of the points of $E$, because the maps vanishing on a fixed $A$ form a
space of cardinality $Q^{n-\rk_{M(E)}(A)}$. This proves~\eqref{app:eq-count}.
For~\eqref{app:eq-QN}, count all linear maps from an $n$-space to $\F_Q$ by the
codimension $h$ of their kernel: there are $\qbinom{n}{h}$ choices for the
kernel and $I_h(Q)$ injective maps from the resulting $h$-dimensional quotient
to $\F_Q$.
\end{proof}

\begin{lemma}[projective-deletion convolution]\label{app:lem-convolution}
Let $V$ be an $r$-dimensional $\F_q$-space, let $D\subseteq\mathbb P(V)$, set
$E=\mathbb P(V)\setminus D$, and let $M$ be the matroid represented by $E$. Define
\begin{equation}\label{app:eq-CD}
C_D(t)=\sum_{\substack{A\le V\\\mathbb P(A)\subseteq D}}t^{\dim A},
\end{equation}
where the zero subspace is included in the sum. For a flat $F$ of $M$, put
$U=\Span F$, and define $D_U\subseteq\mathbb P(V/U)$ by declaring a quotient point
$B/U$ with $\dim(B/U)=1$ to lie in $D_U$ if and only if
\begin{equation}\label{app:eq-DU}
\mathbb P(B)\setminus\mathbb P(U)\subseteq D.
\end{equation}
Then
\begin{equation}\label{app:eq-conv}
t^rC_D(t^{-1})=\sum_{F\in L(M)}\chi_{M|F}(t)\,C_{D_U}(t).
\end{equation}
\end{lemma}

\begin{proof}
For a subspace $U\le V$, the surviving points in $\mathbb P(U)$ form a flat of $M$
precisely when they span $U$; every flat arises uniquely this way, with
$U=\Span F$. In the contraction by this flat, a point of $\mathbb P(V/U)$ is
represented unless every point in its fiber outside $\mathbb P(U)$ was deleted,
so the simplification of $M/F$ is the matroid represented by
$\mathbb P(V/U)\setminus D_U$. Parallel simplification
changes neither the ranked lattice of flats nor the characteristic polynomials
of its intervals nor its Kazhdan--Lusztig polynomial.

We prove~\eqref{app:eq-conv} by evaluating both sides at the infinitely many
integers $Q=q^m$. By Lemma~\ref{app:lem-count}, $\chi_{M|F}(Q)$ counts the maps
$\phi:U\to\F_Q$ nonzero on every surviving point of $\mathbb P(U)$. A subspace
$\bar A\le V/U$ counted by $C_{D_U}(Q)$ has the form $\bar A=B/U$ with
$\mathbb P(B)\setminus\mathbb P(U)\subseteq D$, and its weight $Q^{\dim(B/U)}$ is the number
of extensions of $\phi$ to a linear map $\psi:B\to\F_Q$. Thus the right side
of~\eqref{app:eq-conv}, evaluated at $Q$, counts triples $(U,B,\psi)$ with $U$
spanned by its surviving points and $\psi$ nonzero on those points. Equivalently
it counts pairs $(B,\psi)$ with $B\le V$, $\psi:B\to\F_Q$, and $\psi(x)\ne0$ for
every $x\in\mathbb P(B)\setminus D$: such a pair determines the unique subspace
$U=\Span(\mathbb P(B)\setminus D)$, which gives a flat of $M$ with all points of
$\mathbb P(B)\setminus\mathbb P(U)$ deleted, and restriction of $\psi$ recovers $\phi$.

Now classify these pairs by $A=\ker\psi$; the defining condition is
$\mathbb P(A)\subseteq D$. For a fixed such $A$ with $a=\dim A$, choosing
$h=\dim(B/A)$ gives $\qbinom{r-a}{h}$ choices for $B/A$ and $I_h(Q)$ choices for
the injection $B/A\to\F_Q$, so by~\eqref{app:eq-QN},
$\sum_{h\ge0}\qbinom{r-a}{h}I_h(Q)=Q^{r-a}$. Summing over $A$ shows the right
side of~\eqref{app:eq-conv} equals
$\sum_{A\le V,\,\mathbb P(A)\subseteq D}Q^{r-\dim A}=Q^rC_D(Q^{-1})$, its left side.
Equality at infinitely many $Q$ gives the polynomial identity.
\end{proof}

\begin{lemma}[projective-deletion degree criterion]\label{app:lem-degree}
Use the notation of Lemma~\ref{app:lem-convolution}, and assume the surviving
points span $V$, so that $\rk M=r$. Suppose that for every flat $F$ with
$U=\Span F$ and $\dim U<r$,
\begin{equation}\label{app:eq-degcond}
\deg C_{D_U}<\frac{r-\dim U}{2}.
\end{equation}
Then $P_M(t)=C_D(t)$.
\end{lemma}

\begin{proof}
If $U\subseteq B$ are spans of flats, quotienting first by $U$ and then by
$B/U$ deletes the same directions as quotienting directly by $B$, so
$(D_U)_{B/U}=D_B$; the flats above $F_U$ are precisely the flats of $M/F_U$, and
hence hypothesis~\eqref{app:eq-degcond} is inherited by every contraction.
Because the surviving points span $V$, their images span every quotient $V/U$,
and since $F_U$ spans $U$ we get $\rk(M/F_U)=r-\dim U$. We show $P_{M/F_U}=C_{D_U}$
for all these contractions by induction on their rank. The rank-zero case is
$U=V$, where $C_{D_V}=1$, matching the normalization. Fix $U$ and assume the
claim for every proper contraction above it; by
Lemma~\ref{app:lem-convolution} those have polynomials $C_{D_B}$. Applying
identity~\eqref{app:eq-conv} in $V/U$ and comparing with the defining recursion
for $P_{M/F_U}$---the reciprocal exponents agree by
$\rk(M/F_U)=r-\dim U$, and all nonempty-flat terms agree by induction---the
difference $H(t)=P_{M/F_U}(t)-C_{D_U}(t)$ satisfies
$t^{r-\dim U}H(t^{-1})=H(t)$. Both summands have degree less than
$(r-\dim U)/2$ by the degree axiom and~\eqref{app:eq-degcond}, so the degrees on
the two sides of this reciprocal identity are disjoint unless $H=0$. Taking
$U=0$ proves $P_M=C_D$.
\end{proof}

\begin{theorem}[explicit counterexample]\label{app:thm-counterexample}
There is a simple binary matroid $M$ of rank $7$ on $114$ points for which
\begin{equation}\label{app:eq-PM}
P_M(t)=1+13t+7t^2+t^3,
\end{equation}
and this polynomial has two nonreal zeros.
\end{theorem}

\begin{proof}
Let $V=\F_2^7$ with basis $e_1,\dots,e_7$, put $W_0=\langle e_1,e_2,e_3\rangle$,
and, identifying a nonzero vector with its projective point, define
\[
S=\{e_4,\;e_4+e_5,\;e_4+e_6,\;e_4+e_7,\;e_4+e_5+e_6,\;e_4+e_5+e_7\},\qquad
D=(W_0\setminus\{0\})\cup S,
\]
and let $M$ be the matroid represented by $\mathbb P(V)\setminus D$. The surviving points span $V$ (for instance
$e_5,e_6,e_7$, $e_1+e_5,e_2+e_5,e_3+e_5$, and $e_4+e_6+e_7$ all survive and span
the basis), so $\rk M=7$; $M$ is simple as a restriction of a projective
geometry, and it has $2^7-1-13=114$ points.

We first count the subspaces $A$ with $\mathbb P(A)\subseteq D$. There is one zero
subspace and $13$ one-dimensional subspaces. The sum of two distinct vectors of
$S$ is a nonzero vector in $\langle e_5,e_6,e_7\rangle$, hence not in $D$; the
sum of a vector of $S$ and a nonzero vector of $W_0$ has nonzero $W_0$- and
$e_4$-components, hence is not in $D$. So no subspace of dimension at least two
containing a point of $S$ lies in $D$, and the remaining contained subspaces lie
in $W_0$: there are $\qbinom{3}{2}=7$ of dimension two, one of dimension three,
and none larger. Consequently
\begin{equation}\label{app:eq-CDvalue}
C_D(t)=1+13t+7t^2+t^3.
\end{equation}

It remains to verify~\eqref{app:eq-degcond}. Let $U$ be the span of a flat with
$k=\dim U<7$, and suppose an $h$-dimensional $\bar A\le V/U$ is counted by
$C_{D_U}$, with preimage $B$. Then $\mathbb P(B)\setminus\mathbb P(U)\subseteq D$, and its
left side has
\[
(2^{k+h}-1)-(2^k-1)=2^k(2^h-1)
\]
points, whereas $|D|=13$. If $h\ge(7-k)/2$, the least admissible $h$ gives
\[
\begin{array}{c|rrrrrrr}
k&0&1&2&3&4&5&6\\\hline
\lceil(7-k)/2\rceil&4&3&3&2&2&1&1\\
2^k(2^h-1)&15&14&28&24&48&32&64
\end{array}
\]
and every entry of the last row exceeds $13$, a contradiction. Hence
$\deg C_{D_U}<(7-k)/2$ for every flat span $U$, and
Lemma~\ref{app:lem-degree} with~\eqref{app:eq-CDvalue} proves~\eqref{app:eq-PM}.

Finally, the discriminant of $t^3+7t^2+13t+1$ is
\[
7^2\cdot13^2-4\cdot13^3-4\cdot7^3-27+18\cdot7\cdot13=-268.
\]
A real cubic with negative discriminant has one real root and a nonreal
conjugate pair, so $P_M$ does not have all of its zeros on the negative real
axis.
\end{proof}

\subsection{Summary of the Rethlas discovery process}\label{app:ai-trajectory}
In this section, we summarize the discovery process of Rethlas, which was reconstructed from the run logs and persistent memory of Rethlas. The process is illustrated in Figure~\ref{fig:dic}.

\begin{figure}
	\begin{center}
		\includegraphics[width=\textwidth]{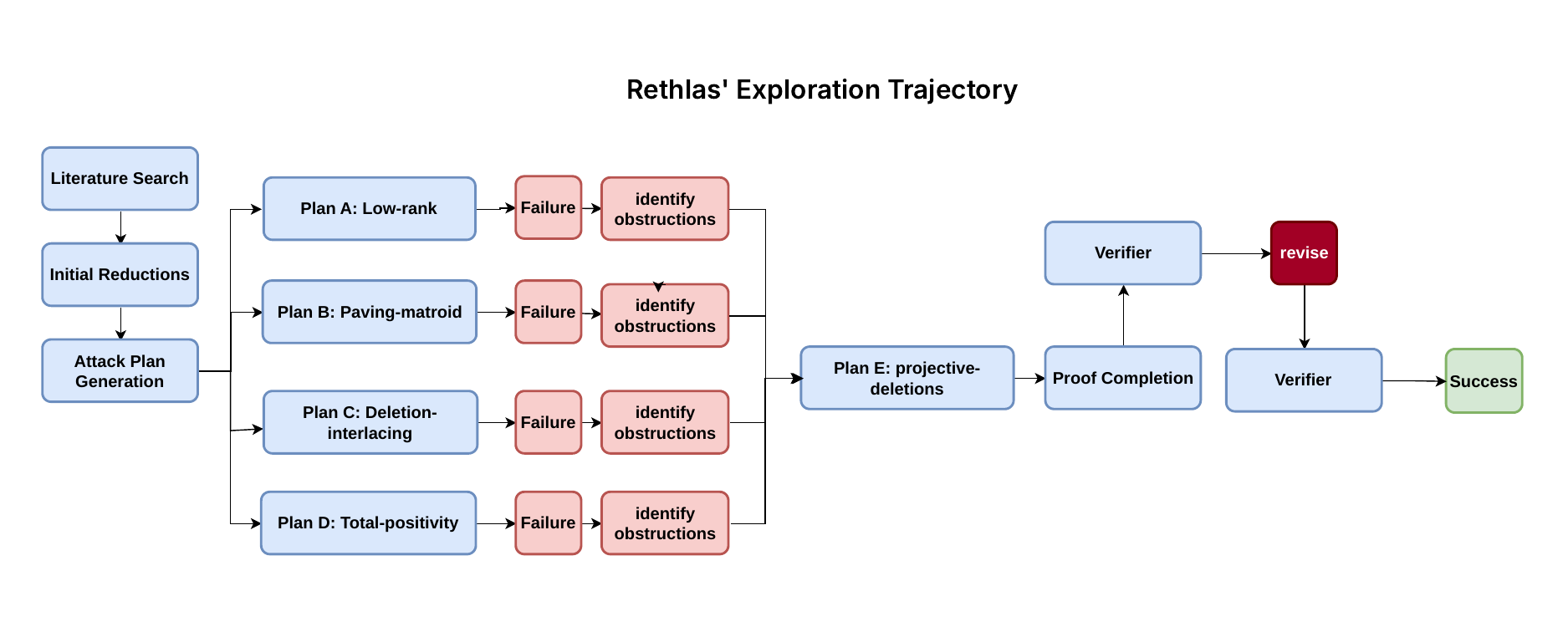}
	\end{center}
	\caption{Discovery process of Rethlas}\label{fig:dic}
\end{figure}

\begin{enumerate}[leftmargin=*]
\item\textbf{Literature and initial reductions.} Rethlas first confirmed that
real-rootedness of matroid Kazhdan--Lusztig polynomials was an open
conjecture~\cite{GPY}. It noted that coefficient
nonnegativity~\cite{BHMPW} excludes nonnegative real zeros and retrieved formulas
for several relevant families.

\item\textbf{Initial approaches and their obstructions.} Rethlas explored four
routes.
\begin{enumerate}[label=\textup{\Alph*.}]
\item\texttt{Low-rank discriminants.} In ranks $5$ and $6$, it reduced real-rootedness of
$P_M(t)=1+at+bt^2$ to the discriminant inequality $a^2\ge4b$; brute-force
searches found no counterexample.
\item\texttt{Paving-matroid relaxation.} Formulas
in~\cite{FerroniNasrVecchi} produced formal
hyperplane counts for which $P_M$ had nonreal zeros, but
Rethlas showed that no matroid realized those counts.
\item\texttt{Deletion-interlacing.} The deletion formula of Braden and
Vysogorets~\cite{BradenVysogorets} contains a negative term that prevents the
standard interlacing argument and leads back to an open contraction-interlacing
problem.
\item\texttt{Total positivity.} Known coefficient-positivity and Hodge-theoretic properties do not by themselves imply real-rootedness, so they could not complete a total-positivity argument.
\end{enumerate}
The paving obstruction was especially informative: prescribing suitable
coefficients was not enough; a better plan may be to construct a counterexample from an explicit
realizable family with computable contractions.

\item\textbf{Projective deletions.} Guided by this requirement, a later round
turned to restrictions of finite projective geometries. These matroids are
representable by construction, and their contractions remain projective
deletions. Deleting a subspace first produced tractable real-rooted families
and exposed the underlying convolution mechanism. Rethlas then allowed more
general sets of deleted points, leading to the convolution identity and degree
criterion reproduced in Section~\ref{app:ai-output}. It first realized
\[
  1+13t+7t^2+t^3
\]
in rank $15$, where the degree bounds hold automatically, and then sharpened
the construction to the rank-$7$ binary matroid in
Theorem~\ref{app:thm-counterexample}.

\item\textbf{Verification and repair.} The Rethlas verification agent rejected
an earlier version because the degree criterion used
$\rk(M/F_U)=r-\dim U$ without a spanning hypothesis. Rethlas added that
hypothesis and verified it for the explicit example. The verification agent
then accepted the revised proof reproduced above.
\end{enumerate}

The rank-$7$ example motivated the general results of this paper.
Theorem~\ref{thm:main-intro} isolates the half-rank degree condition, and
Corollaries~\ref{cor:realization-intro} and~\ref{cor:nonunimodal-intro} produce
nonunimodal examples over every finite field.

\clearpage
\printbibliography

\end{document}